\documentclass[10pt,reqno]{amsart}
\usepackage{amsmath}
\usepackage{amssymb}
\usepackage{amsthm}
\usepackage{eepic,epic}
\usepackage{epsfig}
\usepackage{graphicx}
\usepackage{color}

\newlength{\defbaselineskip}
\newcommand{\setlinespacing}[1]%
           {\setlength{\baselineskip}{#1 \defbaselineskip}}

\numberwithin{equation}{section}

\newtheorem{thm}{Theorem}[section]

\newtheorem{lem}[thm]{Lemma}
\newtheorem{prop}[thm]{Proposition}

\theoremstyle{definition}

\theoremstyle{remark}
\newtheorem{rem}[thm]{Remark}
\numberwithin{equation}{section}

\begin{document}

\title[Strichartz estimates with angular integrability]
{Endpoint Strichartz estimates with angular integrability and some applications}

\author{Jungkwon Kim, Yoonjung Lee and Ihyeok Seo}

\thanks{This research was supported by NRF-2022R1A2C1011312.}


\address{Department of Mathematics, Sungkyunkwan University, Suwon 16419, Republic of Korea}
\email{kimjk809@skku.edu}

\email{yjglee@skku.edu}

\email{ihseo@skku.edu}

\begin{abstract}
The endpoint Strichartz estimate $\|e^{it\Delta} f\|_{L_t^2 L_x^\infty} \lesssim \|f\|_{L^2}$
is known to be false in two space dimensions. 
Taking averages spherically on the polar coordinates $x=\rho\omega$, $\rho>0$, $\omega\in\mathbb{S}^1$, Tao showed a substitute of the form 
$\|e^{it\Delta} f\|_{L_t^2L_\rho^\infty L_\omega^2} \lesssim \|f\|_{L^2}$.  
Here we address a weighted version of such spherically averaged estimates.
As an application, the existence of solutions for the inhomogeneous nonlinear Schr\"odinger equation
is shown for $L^2$ data.
\end{abstract}

\maketitle

\section{Introduction}

The physical interpretation of the Schr\"odinger equation $i\partial_tu+\Delta u=0$
is that $|u(x,t)|^2$ is the probability density for finding
a quantum particle at place $x\in\mathbb{R}^n$ and time $t\in\mathbb{R}$.
This leads us to think that $L^2(\mathbb{R}^n)$ will play a distinguished role.
Indeed, the Schr\"odinger propagator $e^{it\Delta}$, which gives a formula for the solution, is an isometry on $L^2$. That is,
$\|e^{it\Delta}f\|_{L_x^2}=\|f\|_{L^2}$ for any fixed $t$.
But interestingly, when averages on time are also made, a much richer $L^p$ integrability can be observed.
This space-time integrability known as \textit{Strichartz estimates} has been extensively studied over the last several decades and is now completely understood as follows
(see \cite{St,GV4,M-S,KT}):
\begin{equation} \label{Schclas}
\|e^{it\Delta} f\|_{L_t^q L_x^r} \lesssim \|f\|_{L^2}
\end{equation}
if and only if $(q,r)$ is \textit{Schr\"odinger-admissible}, i.e.,
$q\geq 2$, $2/q + n/r = n/2$ and $(q,r,n)\neq(2,\infty,2)$.

The endpoint case $q=2$ is known to be false in two space dimensions (\cite{M-S}),
in which case Tao \cite{T} showed a substitute of the form
\begin{equation*}
\|e^{it\Delta} f\|_{L_t^2L_\rho^\infty L_\omega^2} \lesssim \|f\|_{L^2}
\end{equation*}
by taking averages spherically on the polar coordinates $x=\rho\omega$, $\rho>0$, $\omega\in\mathbb{S}^1$.
Tao's result was extended to higher dimensions
in \cite{GLNY}.
More general spherically averaged estimates
involving $L_t^qL_\rho^r L_\omega^2$
were also studied (see \cite{GLNW,G,GHN}).

\subsection{Endpoint estimates}
In this paper we are concerned with a weighted version of the spherically averaged estimates, which involves weighted mixed norms
with the angular variable treated in a different $L^p$ space than the radial variable:
\begin{equation*}
\|f(x)\|_{L_{\rho}^r L_{\omega} ^{k}(|x|^{-r\gamma})}=
\left(\int_0^\infty \|  \rho ^{-\gamma} f(\rho \omega) \|_{L_{\omega}^{k}(\mathbb{S}^{n-1})}^{r} \rho^{n-1} d\rho \right)^{1/r}
\end{equation*}
where $1\leq r,k\leq \infty$ and $\gamma\geq0$.
Particularly when $r=k$, this norm coincides with the weighted $L^r$ norm, $\|f(x)\|_{L^r(|x|^{-r\gamma})}$.
Our first result is the following endpoint Strichartz estimates with angular integrability.

\begin{thm}\label{prop1}
Let $n \geq 3 $ and $0\leq\gamma\leq1$.
Assume that
\begin{equation}\label{prop_cond0}
\frac1r=\frac{n-2}{2n}+\frac{\gamma}{n} \quad \text{and}\quad \frac{1}{r}-\frac{\gamma}{2(n-1)}\leq\frac{1}{k}\leq \frac1r.
\end{equation}
Then we have
\begin{equation}\label{weighT}	
\| |x|^{-\gamma} e^{it\Delta} f \|_{L_t^{2} L_{\rho}^{r} L_{\omega}^{k}} \lesssim \| f \|_{L^{2}}.
\end{equation}
\end{thm}

\begin{rem}\label{rem}
One can also trivially obtain further estimates \eqref{weighT}
for $(1/r, 1/k)$ contained in the closed quadrangle with vertices $A,D,C,B$ in Figure \ref{fig1},
using the inclusion of $L^k$ spaces on the compact set $\mathbb{S}^{n-1}$.
As will be seen later (Section \ref{sec4}), this trivial region is also needed for
obtaining some applications to nonlinear equations described below.
\end{rem}

We shall give more details about the region of $(1/r, 1/k)$ for which the theorem holds;
the region is given by the closed triangle with vertices $A,E,D$.
Especially when $\gamma\rightarrow0$, $(1/r, 1/k)$ goes to the point $A$ and \eqref{weighT} boils down to the endpoint case $q=2$ of the classical estimates \eqref{Schclas}. The segment $[E,D]$ corresponds to the case $\gamma=1$.
The lower and upper bounds of $1/k$ in \eqref{prop_cond0} determine the segments $[A,E]$ and $[A,D]$, respectively.


\begin{figure}
	\centering	\includegraphics[width=0.75\textwidth]{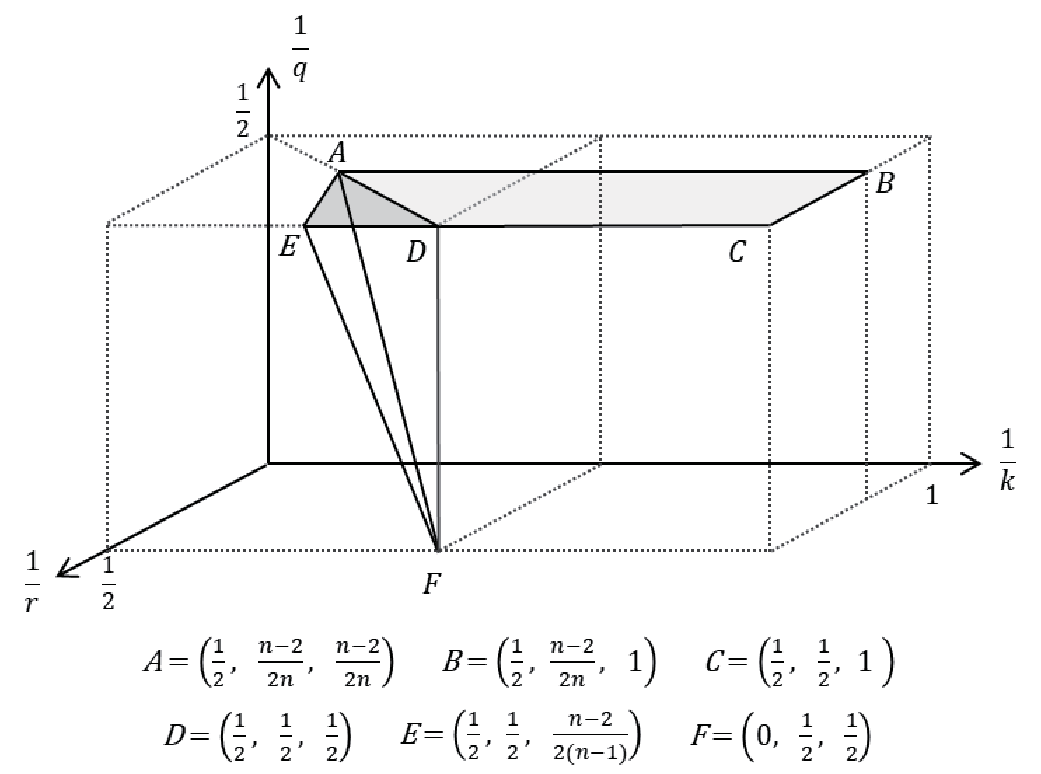}
	 \caption{The range of $(1/r, 1/k)$ for which \eqref{weighT} holds.}
\label{fig1}
\end{figure}


In view of interpolation between \eqref{weighT} and the trivial estimate
$\|e^{it\Delta} f\|_{L_t^\infty L_x^2} \lesssim \|f\|_{L^2}$ (point $F$),
we have
\begin{equation}\label{weighT2}	
\| |x|^{-\gamma} e^{it\Delta} f \|_{L_t^{q} L_{\rho}^{r} L_{\omega}^{k}} \lesssim \| f \|_{L^{2}}
\end{equation}
for $(q,r,k)$ contained in the closed tetrahedron with vertices $A,E,D,F$ with
\begin{equation}\label{gamma-admiss}
\gamma=\frac{2}{q}-n(\frac{1}{2}-\frac{1}{r}).
\end{equation}
This simply recovers the previous result 
of Ozawa and Rogers \cite{OR} in which the non-endpoint case where $q>2$ in \eqref{weighT2} was obtained using Pitt's inequalities, weighted versions of the Hausdorff-Young inequality.
In this regard, the theorem fills the gap $q=2$ in this result.
We also refer the reader to \cite{FW,JWY,CHO} for some related works on more general estimates with angular regularity.
For the optimality of  \eqref{weighT2} particularly when $r=k$, we refer the reader to \cite{LS2}.

\subsection{Applications}
Now we would like to give some applications of the endpoint estimates \eqref{weighT} to the inhomogeneous nonlinear Schr\"odinger equation (INLS)
\begin{equation}\label{INLS}
\begin{cases}
i \partial_{t} u + \Delta u =\lambda  |x|^{-\alpha} |u|^{\beta} u,\quad (x,t) \in \mathbb{R}^n \times \mathbb{R}, \\
u(x, 0)=u_0(x)\in L^2,
\end{cases}
\end{equation}
where $0<\alpha<2$, $\beta>0$ and $\lambda=\pm1$.
Here, the case $\lambda = 1$ is \textit{defocusing}, while the case $\lambda=-1$ is \textit{focusing}.
This model arises in nonlinear optics and plasma physics
for the propagation of laser beams in an inhomogeneous medium (\cite{B,TM}).
The equation enjoys the scale-invariance $u(x,t)\mapsto u_{\delta}(x,t)= \delta^{\frac{2-\alpha}{\beta}} u(\delta x, \delta^2 t)$
for $\delta>0$, and
$$\|u_{{\delta}, 0}\|_{L^2} = \delta^{\frac{2-\alpha}{\beta}-\frac{n}{2}}\|u_0 \|_{L^2}$$
with rescaled initial data $u_{{\delta}, 0}$.
If $\beta=(4-2\alpha)/n$, the scaling preserves the $L^2$ norm of $u_0$
and \eqref{INLS} goes by the name of the mass-critical (or $L^2$-critical) INLS.

This critical case remained unsolved and was recently solved by the authors \cite{KLS}.
(For initial data in other Sobolev spaces $H^s$, see e.g. \cite{GS,Fa,GU,D,CHL,LS}.)
But here, we provide more information on the solution with respect to angular integrability.

To apply the endpoint estimates \eqref{weighT} to the nonlinear problems, one needs inhomogeneous estimates which can be obtained by the standard Christ–Kiselev lemma. However, the lemma misses the following double $L_t^2$ type estimates:

\begin{thm}\label{Prop1}
	Let $n \geq 3 $ and $0<\gamma<1$.
Then we have
	\begin{equation}\label{weighTT*}
		\bigg\| \int_{0} ^{t} e^{i(t-s)\Delta} F(s)ds \bigg\|_{L_t^{2}L_\rho^{r} L_w ^{k}(|\cdot|^{-r\gamma})} \lesssim
		\| F \|_{L_t^{2}L_\rho^{r'} L_w^{\tilde{k}'}(|\cdot|^{r'\gamma})}
	\end{equation}
if 
\begin{equation}\label{prop_cond}
	\frac1r=\frac{n-2}{2n}+\frac{\gamma}{n} 
	\quad\text{and}\quad  \frac{1}{r}-\frac{\gamma}{2(n-1)}<\frac{1}{k}<\frac1r\leq\frac{1}{\tilde{k}}\leq1.
\end{equation}	
\end{thm}

Here, the region of $(1/r,1/k)$ and $(1/r, 1/\tilde{k})$ is given by the open triangle with vertices $A,E,D$ and 
the closed quadrangle with vertices $A,D,C,B$
without the boundaries $[A,B],[D,C]$, respectively.
The weights $|x|^{-\gamma}$ in the weighted estimates allow us to handle the singularity $|x|^{-\alpha}$ in the nonlinearity effectively. 
As a result, we obtain the following well-posedness result.

\begin{thm}\label{cor}
Let $n \geq 3$, $0 < \alpha < 2 $ and $\beta = (4-2\alpha)/n$.
If $\|u_0\|_{L^2}$ is assumed to be small, then there exists a unique solution to \eqref{INLS}
$$u \in C_t([0,\infty) ; L_x^2) \cap L_t^{2}([0,\infty) ; L_{\rho}^{r}L_{\omega}^{k}(|x|^{-\alpha r/2}))$$
for $(r,k)$ satisfying 
\begin{equation}\label{condition22}
 \frac{1}{r}=\frac{n-2}{2n}+\frac{\alpha}{2n}
 \quad\text{and}\quad \frac{1}{r}-\frac{\alpha}{4(n-1)}<\frac{1}{k} <\frac{1}{r}.
\end{equation}	
Furthermore, the solution scatters in $L^2$, i.e., there exists $\varphi\in L^2$ such that
$$\lim_{t\rightarrow\infty}\|u(t)-e^{it\Delta}\varphi\|_{L_x^2}=0.$$
\end{thm}

We note in passing that similar results can be obtained by treating the non-endpoint estimates \eqref{weighT2}.
However, we do not pursue this issue here.

The paper is organized as follows.
In Section \ref{sec2}
we prove the endpoint estimates in Theorem \ref{prop1}
by making use of interpolation together with the Sobolev embedding on the unit sphere $\mathbb{S}^{n-1}$.
Section \ref{sec3} is devoted to proving the inhomogeneous estimates in Theorem \ref{Prop1}.
Here we basically adopt the bilinear interpolation argument developed by Keel and Tao \cite{KT},
but we need to modify the argument to make it applicable to the weighted setting.
In Section \ref{sec4} we finally apply the estimates to obtain Theorem \ref{cor}.

Throughout this paper, the letter $C$ stands for a positive constant which may be different
at each occurrence. We also denote $A\lesssim B$ to mean $A\leq CB$
with unspecified constants $C>0$.


\section{Endpoint estimates}\label{sec2}

In this section we prove Theorem \ref{prop1}
by interpolating between the estimates on the point $A$ and the segment $[E,D]$.
As mentioned below Remark \ref{rem}, the estimate on the point $A$ becomes equivalent to the endpoint case $q=2$ of the classical estimates \eqref{Schclas}. Hence we only need to obtain \eqref{weighT} on the segment $[E,D]$;
this is the case $r=2,\gamma=1$ corresponding to 
\begin{equation*}
	\||x|^{-1}e^{it\Delta} f\|_{L_t^2 L_{\rho}^2 L_{\omega}^k} \lesssim \| f \|_{L^2}
\end{equation*}
for $\frac{n-2}{2(n-1)} \leq \frac{1}{k} \leq \frac{1}{2}$.
From \cite{H} (see (1.2) there), we first recall 
\begin{equation}\label{we}
	\|\Lambda^{\frac{1}{2}}|x|^{-1}e^{it\Delta} f\|_{L_t^2 L_{x}^2} \lesssim \| f\|_{L^2}
\end{equation}
where $\Lambda :=\sqrt{1-\Delta_{\omega}}$ for the Laplace-Beltrami operator $\Delta_{\omega}$ on the unit sphere $\mathbb{S}^{n-1}$, $n\ge3$.
Then by \eqref{we}, it is enough to show
\begin{equation}\label{bnm}
	\||x|^{-1}e^{it\Delta} f\|_{L_t^2 L_{\rho}^2 L_{\omega}^k} \lesssim \|\Lambda^{\frac{1}{2}}|x|^{-1}e^{it\Delta} f\|_{L_t^2 L_{\rho}^2 L_{\omega}^2}
\end{equation}
for $\frac{n-2}{2(n-1)} \leq \frac{1}{k} \leq \frac{1}{2}$.
For this, we first apply the Sobolev embedding\footnote{See, for example, Lemma 7.1 in \cite{JWY}} on $\mathbb{S}^{n-1}$, 
\begin{equation}\label{em}
	\|f\|_{L_{\omega}^p} \lesssim \| \Lambda^{(n-1)(\frac{1}{2}-\frac{1}{p})}f\|_{L_{\omega}^2}, \quad 2\leq p<\infty,
\end{equation}
to see
\begin{equation}\label{bnm3}
	\||x|^{-1}e^{it\Delta} f\|_{L_t^2 L_{\rho}^2 L_{\omega}^k} \lesssim \|\Lambda^{(n-1)(\frac{1}{2}-\frac{1}{k})} |x|^{-1}e^{it\Delta} f\|_{L_t^2 L_{\rho}^2 L_{\omega}^2}
\end{equation}
for $2 \leq k < \infty$.
We then use the inclusion $L_{\omega}^{p} \subseteq L_{\omega}^2$ for $p\ge2$ and 
\eqref{em} again to get
\begin{align}\label{bnm2}
	\|\Lambda^{(n-1)(\frac{1}{2}-\frac{1}{k})} g\|_{ L_{\omega}^2}
	\lesssim \|\Lambda^{(n-1)(\frac{1}{2}-\frac{1}{k})} g\|_{ L_{\omega}^{p}}
	\lesssim \|\Lambda^{\frac{1}{2}} g\|_{L_{\omega}^{2}}
\end{align}
with $\frac{1}{p} = \frac{2n-3}{2(n-1)} - \frac{1}{k}$ and $2\leq p<\infty$.
Now the desired estimate \eqref{bnm} follows from
combining \eqref{bnm3} and \eqref{bnm2} with $g=|x|^{-1}e^{it\Delta} f$.
Note that the conditions on $k$ and $p$ here
determine $\frac{n-2}{2(n-1)} \leq \frac{1}{k} \leq \frac{1}{2}$ we wanted.

\section{Inhomogeneous estimates}\label{sec3}

Next we prove the inhomogeneous estimates \eqref{weighTT*} in Theorem \ref{Prop1}
by adopting the bilinear interpolation argument developed by Keel and Tao \cite{KT}.
We need to modify the argument to make it applicable to the present setting with weights.

Instead of \eqref{weighTT*}, we shall show a stronger estimate which is given by replacing $\int_{0}^t$ in \eqref{weighTT*} by $\int_{-\infty}^t$:
	\begin{equation}\label{proto_inhomo0}
	\bigg\| \int_{-\infty} ^{t} e^{i(t-s)\Delta} F(\cdot,s)ds \bigg\|_{L_t^{2}L_\rho^{r} L_\omega ^{k}(|\cdot|^{-r\gamma})}
	\lesssim\| F \|_{L_t^{2}L_\rho^{r'} L_\omega^{\tilde{k}'}(|\cdot|^{r' \gamma})}
\end{equation}
for $r, k, \tilde{k}, \gamma$ given as in the theorem.
To deduce \eqref{weighTT*} from \eqref{proto_inhomo0},
first decompose the $L_t^2$ norm in the left-hand side of \eqref{weighTT*}
into two parts, $t\geq0$ and $t<0$. Then the latter can be reduced to the former
by a change of variables $t\mapsto-t$, and so we only need to consider the first part $t\geq0$.
But, since $[0,t)=(-\infty,t)\cap[0,\infty)$, by applying \eqref{proto_inhomo0} with $F$ replaced by $\chi_{[0,\infty)}(s)F$,
the first part follows directly.

Then by duality, we may show the following bilinear form estimate
\begin{equation}\label{endpt}
		| T(F,G) | \lesssim \left\|F \right\|_{L_t^{2} L_\rho^{r'} L_{\omega}^{k'}(|\cdot|^{r'\gamma})}  \left\| G \right\|_{L_t^{2} L_\rho^{r'} L_{\omega}^{k'}(|\cdot|^{r'\gamma})}
	\end{equation}
where
$$ T(F, G) := \int_{-\infty}^\infty \int_{s<t} \left\langle e^{-i s \Delta} F(s), e^{-i t \Delta} G(t) \right\rangle_{x}  ds dt.$$
Here, $\langle \cdot\,,\cdot\rangle$ denotes the usual inner product on $L^2$.
Indeed, by duality, \eqref{endpt} implies
\eqref{proto_inhomo0} with $\tilde{k}'$ replaced by $k'$. Since $k'\leq\tilde{k}'$, \eqref{proto_inhomo0} then follows.  
By duality and symmetry, we also note that \eqref{endpt} implies  
\begin{equation*}
	\bigg\| \int_{-\infty} ^{\infty} e^{i(t-s)\Delta} F(\cdot,s)ds \bigg\|_{L_t^{2}L_\rho^{r} L_\omega ^{k}(|\cdot|^{-r\gamma})}
	\lesssim\| F \|_{L_t^{2}L_\rho^{r'} L_\omega^{k'}(|\cdot|^{r' \gamma})}
\end{equation*}
with the same $\gamma,r,k$, from which we can alternatively prove \eqref{weighT} on the open triangle with vertices $A,E,D$ in Figure \ref{fig1} using the $TT^*$ argument.

Let us now show \eqref{endpt}. We first decompose the integral region
$\Omega=\{ (s, t)\in \mathbb{R}^2 :s<t \}$
dyadically away from the singularity $t=s$.
Indeed, we break $\Omega$ into a series of time-localized regions using a Whitney type decomposition (see \cite{S} or \cite{F});
let $\mathcal{Q}_j$ be the family of dyadic squares in $\Omega$ whose side length is dyadic number $2^j$ for $j\in \mathbb{Z}$.
Each square $Q =I\times J \in \mathcal{Q}_j$ has the property that
\begin{equation}\label{whitney}
 2^j\sim |I| \sim |J| \sim \textnormal{dist}(I, J)
\end{equation}
and $\Omega = \cup_{j\in \mathbb{Z}} \cup_{Q \in \mathcal{Q}_j}Q$ where the squares $Q$ are essentially disjoint.
Now we may write
	\begin{equation*}
	T(F, G) =\sum_{j \in \mathbb{Z}} T_j(F,G),
	\end{equation*}
where
	\begin{equation*}
		T_j(F,G) : = \sum_{Q \in \mathcal{Q}_j} \int_{t \in J} \int_{s \in I} \left\langle e^{-i s \Delta} F( s), e^{-i t \Delta} G( t) \right\rangle_x  ds dt.
	\end{equation*}
We then obtain the desired estimate \eqref{endpt} by making use of the bilinear interpolation between its time-localized estimates in the following proposition which will be proved later.

\begin{prop} \label{homo_local}
Let $n \ge 3$ and $0 < \gamma<1$. Assume that $2 \leq a, \tilde{a} < \infty$,
\begin{equation}\label{condi_local}
\frac{1}{a} -\frac{\gamma}{2(n-1)} \leq \frac{1}{b} \leq \frac{1}{a} \quad \textnormal{and} \quad
\frac{1}{\tilde{a}}-\frac{\gamma}{2(n-1)} \leq \frac{1}{\tilde{b}} \leq \frac{1}{\tilde{a}}.
\end{equation}
Then we have
	\begin{equation}\label{local}
	|T_j (F,G)| \lesssim 2^{-j\beta(a,\tilde{a})}\|F\|_{L_t^{2} L_\rho^{a'}L_{\omega}^{b'}(|\cdot|^{a'\gamma}) } \|G\|_{L_t^{2} L_\rho^{\tilde{a}'}L_{\omega}^{\tilde{b}'}(|\cdot|^{\tilde{a}'\gamma})}
	\end{equation}
for all $j \in \mathbb{Z}$ and all $(\frac {1}{a}, \frac {1}{\tilde{a}})$ in a neighborhood of $(\frac{1}{r}, \frac{1}{r})$ (see Figure \ref{fig2})
with $$\frac1r = \frac{n-2}{2n}+\frac{\gamma}{n}\quad\text{and}\quad\beta(a,\tilde{a})= -1 + \frac{n}{2} - \frac{n}{2a} - \frac{n}{2\tilde{a}} +\gamma.$$
\end{prop}

\begin{figure}
	\centering
	\includegraphics[width=0.4\textwidth]{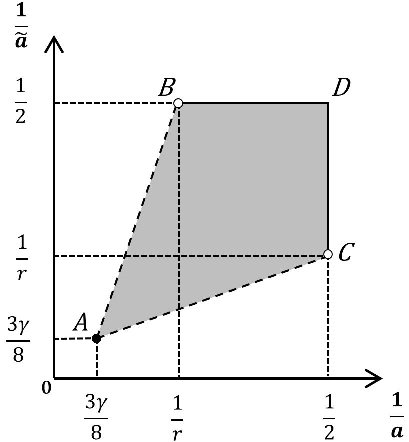}
	\caption{The range of $(1/a, 1/\tilde{a})$ for which \eqref{local} holds.}
\label{fig2}
\end{figure}

From making use of the bilinear interpolation between the estimates \eqref{local}, we shall now deduce
\begin{equation}\label{aj}
\sum_{j \in \mathbb{Z}}| T_{j}(F,G) | \lesssim \|F\|_{L_t^{2} L_\rho^{r'} L_{\omega}^{k'}(|\cdot|^{r'\gamma})}
\| G \|_{L_t^{2} L_\rho^{r'} L_{\omega}^{k'}(|\cdot|^{r'\gamma})}
	\end{equation}
which clearly implies \eqref{endpt}.

Indeed, from the proposition we have the following three estimates
	\begin{align*}
	|T_j (F,G)|  \lesssim 2^{-j\beta(r_0,r_0)}& \|F\|_{L_t^{2} L_\rho^{r_{0}'}L_{\omega}^{k_{0}'}(|\cdot|^{r_{0}'\gamma}) } \|G\|_{L_t^{2} L_\rho^{r_{0}'}L_{\omega}^{k_{0}'}(|\cdot|^{r_{0}'\gamma})}, \cr
	|T_j (F,G)|  \lesssim 2^{-j\beta(r_0,r_1)}& \|F\|_{L_t^{2} L_\rho^{r_{0}'}L_{\omega}^{k_{0}'}(|\cdot|^{r_{0}'\gamma}) } \|G\|_{L_t^{2} L_\rho^{r_{1}'}L_{\omega}^{k_{1}'}(|\cdot|^{r_{1}'\gamma})}, \cr
	|T_j (F,G)|  \lesssim 2^{-j\beta(r_1,r_0)}& \|F\|_{L_t^{2} L_\rho^{r_{1}'}L_{\omega}^{k_{1}'}(|\cdot|^{r_{1}'\gamma}) } \|G\|_{L_t^{2} L_\rho^{r_{0}'}L_{\omega}^{k_{0}'}(|\cdot|^{r_{0}'\gamma})},
	\end{align*}	
where, for a sufficiently small $\varepsilon >0$ and $i=0,1$,
\begin{equation}\label{qdg}
\frac{1}{r_0}=\frac{n-2}{2n}+\frac{\gamma}{n} -\varepsilon,\quad \frac{1}{r_{1}}=\frac{n-2}{2n}+\frac{\gamma}{n}+2\varepsilon
\quad\text{and}\quad
\frac{1}{r_i}-\frac{\gamma}{2(n-1)} \leq \frac{1}{k_i} \leq \frac{1}{r_i}.
\end{equation}
Next we define the vector-valued bilinear operator $B$ by
\begin{equation*}
B(F,G) = \left\{ T_{j} (F,G) \right\}_{j \in \mathbb{Z}}.
\end{equation*}
Then the above three estimates are rewritten as
\begin{equation*}
	\begin{aligned}
	\left\|B (F,G)\right\|_{\ell_{\infty}^{\beta_{0}}} \lesssim & \|F\|_{L_t^{2} L_\rho^{r_{0}'}L_{\omega}^{k_{0}'}(|\cdot|^{r_{0}'\gamma}) } \|G\|_{L_t^{2} L_\rho^{r_{0}'}L_{\omega}^{k_{0}'}(|\cdot|^{r_{0}'\gamma})}, \cr
	\left\|B (F,G)\right\|_{\ell_{\infty}^{\beta_{1}}} \lesssim & \|F\|_{L_t^{2} L_\rho^{r_{0}'}L_{\omega}^{k_{0}'}(|\cdot|^{r_{0}'\gamma}) } \|G\|_{L_t^{2} L_\rho^{r_{1}'}L_{\omega}^{k_{1}'}(|\cdot|^{r_{1}'\gamma})}, \cr
	\left\|B (F,G)\right\|_{\ell_{\infty}^{\beta_{1}}} \lesssim & \|F\|_{L_t^{2} L_\rho^{r_{1}'}L_{\omega}^{k_{1}'}(|\cdot|^{r_{1}'\gamma}) } \|G\|_{L_t^{2} L_\rho^{r_{0}'}L_{\omega}^{k_{0}'}(|\cdot|^{r_{0}'\gamma})},
	\end{aligned}
\end{equation*}	
respectively, with $\beta_{0} = \beta(r_0,r_0)$ and $\beta_{1} =\beta(r_0,r_1)=\beta(r_1,r_0)$.
Here, $\ell_{q} ^{s}$ denotes a weighted sequence space defined for $s \in \mathbb{R}$ and $1\leq q \leq \infty$
with the norm
\begin{eqnarray*}
		\|\{x_j\}_{j \ge 0}\|_{\ell_{q}^{s}}=
		\begin{cases}
			\big(\sum_{j \ge0} 2^{jsq} |x_{j}|^{q}\big)^{1/q} \quad \text{if}\quad q \neq \infty,\\
			 \sup_{j \ge0} 2^{js}|x_{j}| \quad \text{if}\quad q=\infty.
		\end{cases}
\end{eqnarray*}
Applying the following lemma with $p=q=2$ and $\theta_0=\theta_1=1/3$, we now get
	\begin{equation}\label{fhk}
	B  :  (A_{0}, A_{1})_{\frac{1}{3},2}\times(B_{0},B_{1})_{\frac{1}{3},2}
\rightarrow\big(\ell_{\infty} ^{\beta_0}, \ell_{\infty} ^{\beta_1}\big)_{\frac{2}{3},1}
	\end{equation}
with $A_0=B_0= L_t^{2} L_\rho^{r_{0}'}L_{\omega}^{k_{0}'}(|\cdot|^{r_{0}'\gamma})$ and
$A_1=B_1=L_t^{2} L_\rho^{r_{1}'}L_{\omega}^{k_{1}'}(|\cdot|^{r_{1}'\gamma})$.

\begin{lem}[\cite{BL}, Section 3.13, Exercise 5(b)]\label{biinter}
	For $i=0,1$, let $A_i , B_i , C_i$ be Banach spaces and let $T$ be a bilinear operator such that $T: A_0 \times B_0 \rightarrow C_0$, $T : A_0 \times B_1 \rightarrow C_1$, and $T : A_1 \times B_0 \rightarrow C_1$. Then one has
	\begin{equation*}
	T : (A_0,A_1)_{\theta_0, p} \times (B_0, B_1)_{\theta_1, q} \rightarrow (C_0, C_1)_{\theta,1}
	\end{equation*}
	if $0< \theta_i < \theta=\theta_0+\theta_1 <1$ and $1/p + 1/q \geq 1$ for $1 \leq p,q \leq \infty$.
Here, $(\cdot\,,\cdot)_{\theta,p}$ denotes the real interpolation functor.
\end{lem}

Finally, we shall apply the real interpolation space identities in the following lemma (see \cite{C} and \cite{BL}).
\begin{lem}
Let $0<\theta<1$.
If $1\leq p_0, p_1 <\infty$ and
$1/p=(1-\theta)/p_{0} + \theta/p_{1}$, then	
\begin{equation}\label{33}
 (L^{p_0}(A_0), L^{p_1}(A_1))_{\theta, q} =
 \begin{cases}
L^{p}((A_0, A_1)_{\theta, p}) \quad\text{if}\quad q=p, \\
L^{p,q}(A) \quad\text{if}\quad A_0 =A_1=A,
\end{cases}
\end{equation}
for Banach spaces $A_0,A_1$.
If $s_0,s_1\in\mathbb{R}$, $s_{0} \neq s_{1}$ and $s=(1-\theta)s_{0} +\theta s_{1}$, then
	\begin{equation*}
	(\ell_{\infty}^{s_{0}}, \ell_{\infty}^{s_{1}})_{\theta, 1} = \ell_1^{s}.
	\end{equation*}
\end{lem}
Indeed, applying the lemma, we first see
$
(\ell_{\infty} ^{\beta_0}, \ell_{\infty} ^{\beta_1})_{\frac{2}{3},\,1} = \ell_{1} ^{0}
$
and
\begin{equation*}
\big(L_t^{2} L_\rho^{r_{0}'} L_{\omega}^{k_{0}'}(|\cdot|^{r_{0}' \gamma }), L_t^{2} L_\rho^{r_{1}'} L_{\omega}^{k_{1}'}(|\cdot|^{ r_{1}' \gamma})\big)_{\frac{1}{3},\,2}
= L_t^{2} \big(\big(L_\rho^{r_0'} L_{\omega}^{k_0'}(|\cdot|^{r_0'\gamma }), L_\rho^{r_1'} L_{\omega}^{k_1'}(|\cdot|^{r_1'\gamma })\big)_{\frac{1}{3},\, 2} \big).
\end{equation*}
Since $r' < 2$, the first identity in \eqref{33} does not applied inside the $L_t^2$ space any more.
Instead we will make use of the second one,
and hence we must take $k'=k_0'=k_1'$, but this is possible because \eqref{qdg} holds for a sufficiently small $\varepsilon>0$
if $$\frac{1}{r}-\frac{\gamma}{2(n-1)}<\frac{1}{k}<\frac{1}{r}$$
which is exactly consistent with that in the assumption \eqref{prop_cond}.
Since we may write
$$\|f\|_{L_{\rho}^{r'} L_\omega^{k'} (|\cdot|^{r'\gamma})}=\|F\|_{L_{\tilde{\rho}}^{r'}L_\omega^{k'}}$$
with $\tilde{\rho}=\rho^n/n$ and $F(\tilde{\rho}, \omega)=(n\tilde{\rho})^{\gamma/n} f((n \tilde{\rho})^{1/n} \omega)$,
we are indeed reduced to showing
$$L_t^{2} \big( \big(L_{\tilde{\rho}}^{r_0'} L_{\omega}^{k'}, L_{\tilde{\rho}}^{r_1'} L_{\omega}^{k'}\big)_{\frac{1}{3},\, 2} \big)=
L_t^{2}L_{\tilde{\rho}}^{r',2} L_{\omega}^{k'}\supset L_t^2 L_{\tilde{\rho}}^{r'}L_{\omega}^{k'}$$
which follows directly from applying the second identity in \eqref{33} and then using the embedding property of Lorentz spaces,
$L_{\tilde{\rho}}^{r'} \subset L_{\tilde{\rho}}^{r', 2}$ for $r'< 2$.
Combining \eqref{fhk} with the resulting real interpolation spaces, we now get
 \begin{align*}
 	B :  L_t^{2} L_\rho^{r'} L_{\omega}^{k'}(|\cdot|^{r'\gamma })  \times L_t^{2} L_\rho^{r'} L_{\omega}^{k'}(|\cdot|^{r'\gamma })
 \rightarrow \ell_{1} ^{0}
 \end{align*}
 which is equivalent to the desired estimate \eqref{aj}. This completes the proof.

\subsection{Time-localized estimates}
This subsection is devoted to proving the time-localized estimate \eqref{local} in Proposition \ref{homo_local}.
Let us first set
\begin{equation}\label{cnm}
T_{j,Q}(F,G):=\int_{t \in J} \int_{s \in I} \big\langle e^{-i s \Delta} F( s),  e^{-i t \Delta} G( t) \big\rangle_x ds dt
\end{equation}
for each square $Q =I\times J \in \mathcal{Q}_j$.
Then we only need to show
\begin{align}\label{cube_sum}
|T_{j,Q}(F, G)| \lesssim 2^{-j\beta( a,\tilde{a})}
\|F\|_{L_t^{2}(I; L_\rho^{a'}L_{\omega}^{b'}(|\cdot|^{a'\gamma}))}
\|G\|_{L_t^{2}(J; L_\rho^{\tilde{a}'}L_{\omega}^{\tilde{b}'}(|\cdot|^{\tilde{a}'\gamma}))}
\end{align}
to get \eqref{local}.
Using the fact that for each $I$ there are at most a fixed finite number of intervals $J$ which satisfy \eqref{whitney} and they are all contained
in a neighborhood of $I$ of size $O(2^j)$, we indeed get
\begin{align*}
 \sum_{Q \in \mathcal{Q}_j} \left|T_{j,Q}(F, G) \right| \lesssim &2^{-j\beta(a,\tilde{a})}\sum_{Q \in \mathcal{Q}_j}
\|F\|_{L_t^{2}(I; L_\rho^{a'}L_{\omega}^{b'}(|\cdot|^{a'\gamma}))}
\|G\|_{L_t^{2}(I; L_\rho^{\tilde{a}'}L_{\omega}^{\tilde{b}'}(|\cdot|^{\tilde{a}'\gamma}))} \cr
 \leq  2^{-j\beta(a, \tilde{a})} &\Big(\sum_{Q \in \mathcal{Q}_j }
\|F \|_{L_t^{2}(I; L_\rho^{a'}L_{\omega}^{b'}(|\cdot|^{a'\gamma}))}^{2}  \Big)^{\frac{1}{2}}
\cdot \Big(\sum_{Q \in \mathcal{Q}_j}
\|G\|_{L_t^{2}(I; L_\rho^{\tilde{a}'}L_{\omega}^{\tilde{b}'}(|\cdot|^{\tilde{a}'\gamma}))}^{2} \Big)^{\frac{1}{2}}\cr
\lesssim  2^{-j\beta(a, \tilde{a})}&
\|F\|_{L_t^{2}(\mathbb{R}; L_\rho^{a'}L_{\omega}^{b'}(|\cdot|^{a'\gamma}))}
\|G\|_{L_t^{2}(\mathbb{R}; L_\rho^{\tilde{a}'}L_{\omega}^{\tilde{b}'}(|\cdot|^{\tilde{a}'\gamma}))}
\end{align*}
as desired.

From now on, we shall show \eqref{cube_sum} for the following exponents (see Figure \ref{fig2}):
\begin{itemize}
  \item [$(a)$] $a = \tilde{a} = \frac8{3\gamma}:=\lambda$,\quad$b=\tilde{b}$\quad (point $A$),
  \item [$(b) $] $2 \leq a < r=\frac{2n}{n-2+2\gamma},\quad \tilde{a}=2$ \quad(segment $(B,D]$),
  \item [$(c) $] $a=2,\quad  2 \leq \tilde{a} < r=\frac{2n}{n-2+2\gamma}$ \quad(segment $(C,D]$),
\end{itemize}
in which $b$ and $\tilde{b}$ are also given to hold \eqref{condi_local}.
The proposition will then follow by interpolation and the fact that $2<r<\infty$.

To show the first case $(a)$, we recall the following time decay estimates (see Proposition 4.2 in \cite{OR}):
\begin{equation}\label{weighted_decay}
\left\| e^{it\Delta} u_0 \right\|_{L_{\rho}^{a} L_{\omega}^{b}(|\cdot|^{-a\gamma})} \lesssim |t|^{-n( \frac{1}{2}- \frac{1}{a})-\gamma} \left\| u_0 \right\|_{L_{\rho}^{a'}L_{\omega}^{b'}(|\cdot|^{a'\gamma})}
\end{equation}
where $ 2 \leq a \leq b < \infty$ and $2(n-1)(\frac{1}{a}-\frac{1}{b}) \leq \gamma < \frac{n}{a}$.
Since this estimate does not hold for $a=\infty$, we cannot take the origin instead of the point $A$ in Figure \ref{fig2}.
Hence we need to carefully choose the point $A$ near the origin by observing, from the condition $\gamma< \frac{n}{a}$,
the fact that
the more nearer we take the point $A$ to the origin, the higher the admissible dimension is.
The point $A=(\frac{3\gamma}{8},\frac{3\gamma}{8})$ would suffice to cover all dimensions $n \ge 3$.
Now we use H\"{o}lder's inequality and \eqref{weighted_decay} to obtain
	\begin{align*}
	|T_{j,Q} (F,G)| &\leq \int_{J} \int_{I} \| \rho^{-\gamma}e^{i(t-s) \Delta} F(s) \|_{L_\rho^{\lambda} L_{\omega}^{b}} \| \rho^{\gamma} G(t) \|_{L_\rho^{\lambda'} L_{\omega}^{{b}'}}\,  ds dt \\
& \lesssim \int_{J}  \int_{I} |t-s|^{-n(\frac{1}{2}-\frac{1}{\lambda})-\gamma} \left\| F(s) \right\|_{L_\rho^{\lambda'} L_{\omega}^{{b}'}(|\cdot|^{\lambda'\gamma})} \left\| G(t) \right\|_{L_\rho^{\lambda'} L_{\omega}^{{b}'}(|\cdot|^{\lambda'\gamma})}dsdt \\
&\lesssim 2^{-jn(\frac{1}{2}-\frac{1}{\lambda})-j\gamma} \int_{J} \int_{I}  \|F(s)\|_{L_\rho^{\lambda'} L_{\omega}^{{b}'}(|\cdot|^{\lambda'\gamma})} \|G(t)\|_{L_\rho^{\lambda'} L_{\omega}^{{b}'}(|\cdot|^{\lambda'\gamma})}dsdt \\
&\leq 2^{-jn(\frac{1}{2}-\frac{1}{\lambda})-j\gamma}  \int_{I} \| F \|_{L_\rho^{\lambda'} L_{\omega}^{{b}'}(|\cdot|^{\lambda'\gamma})} ds \cdot \int_{J} \| G \|_{L_\rho^{\lambda'} L_{\omega}^{{b}'}(|\cdot|^{\lambda'\gamma})}dt,
\end{align*}
where $\lambda \leq b < \infty$ and $2(n-1)(1/\lambda-1/b) \leq \gamma $ are required. 
This requirement is the same as the condition \eqref{condi_local} with $a=\lambda$.
 By applying H\"{o}lder's inequality again in each $t$ and $s$, we get
	\begin{equation*}
	|T_{j,Q} (F,G)| \lesssim 2^{-j\beta(\lambda,\lambda)} \|F\|_{L_t^2(I; L_\rho^{\lambda'} L_{\omega}^{{b}'}(|\cdot|^{\lambda'\gamma}))} \|G\|_{L_t^2(J; L_\rho^{\lambda'} L_{\omega}^{{b}'}(|\cdot|^{\lambda'\gamma}))},
	\end{equation*}
as desired.

Now it remains to show the second case $(b)$. (The case $(c)$ is shown clearly in the same way.)
By bringing the $s$-integration inside the inner product in \eqref{cnm}
and then applying H\"{o}lder's inequality in $\omega$, $\rho$ and $t$ in turn,
we first see that
	\begin{align}\label{second_case}
	\nonumber|T_{j,Q}(F, G)|
	&\leq \int_{J} \left| \left\langle \int_{I}  e^{i(t-s)\Delta}F( s) ds , G( t) \right\rangle_{x} \right| dt \\
	&\leq  \left\|  \int_{\mathbb{R}}  e^{i(t-s)\Delta} \chi_{I}(s) F(s) ds \right\|_{L_t^{\tilde{q}} L_{\rho}^{2} L_{\omega}^{\tilde{b}}(|\cdot|^{-2\gamma})} \| G \|_{L_t^{\tilde{q}'} (J; L_{\rho}^{2} L_{\omega}^{\tilde{b}'}(|\cdot|^{2\gamma}))}
	\end{align}	
where  $\tilde{q} >2$ is given so that \eqref{gamma-admiss} holds for $(q,r)=(\tilde{q},2)$.
Then by using the $TT^*$ version of \eqref{weighT2}, we have
\begin{equation}\label{TT*}
	\bigg\| \int_{\mathbb{R}} e^{i(t-s) \Delta} \chi_{I}(s)  F( s)ds \bigg\|_{L_t^{\tilde{q}}L_{\rho}^{2}L_{\omega}^{\tilde{b}}(|\cdot|^{-2\gamma})} \lesssim \| F \|_{L_t^{{q}'}(I;L_{\rho}^{a'}L_{\omega}^{{b}'}(|\cdot|^{a'\gamma}))}
\end{equation}
for $q >2$ given so that \eqref{gamma-admiss} holds for $(q,r)=(q,a)$.
(Note here that there can exist such $q,\tilde{q}>2$ for $0<\gamma<1$.)
By combining \eqref{second_case} and \eqref{TT*}, we now get
\begin{equation*}
 |T_{j,Q}(F, G)| \lesssim \| F \|_{L_t^{{q}'}(I;L_{\rho}^{a'}L_{\omega}^{{b}'}(|\cdot|^{a'\gamma}))}
\| G \|_{L_t^{\tilde{q}'} (J; L_{\rho}^{2} L_{\omega}^{\tilde{b}'}(|\cdot|^{2\gamma}))}.
\end{equation*}	
Using H\"older's inequality in $t$ since $q'<2$, and then using the identity \eqref{gamma-admiss}
for $(q,r)=(q,a)$, we estimate
\begin{align*}
 \left\|  F  \right\|_{L_t^{q'}(I; L_{\rho}^{a'} L_{\omega}^{b'}(|\cdot|^{a'\gamma}))}
 &\lesssim 2^{j(\frac{1}{2}-\frac{1}{q})} \| F \|_{L_t^{2} (I; L_{\rho}^{a'} L_{\omega}^{b'}(|\cdot|^{a'\gamma}))} \cr
 &= 2^{ j\left(\frac{1}{2}-\frac{n}{2}\left( \frac{1}{2}-\frac{1}{a} \right) -\frac{\gamma}{2}\right) } \| F \|_{L_t^{2} (I; L_{\rho}^{a'} L_{\omega}^{b'}(|\cdot|^{a'\gamma}))}
 \end{align*}
and similarly
 \begin{align*}
\| G \|_{L_t^{\tilde{q}'} (J; L_{\rho}^{2} L_{\omega}^{\tilde{b}'}(|\cdot|^{2\gamma}))} \lesssim 2^{j(\frac{1}{2}-\frac{\gamma}{2})} \| G\|_{L_t^{2} (J; L_{\rho}^{2} L_{\omega}^{\tilde{b}'}(|\cdot|^{2\gamma}))}.
 \end{align*}
Therefore, we get
$$|T_{j,Q}(F, G)| \lesssim 2^{-j\beta(a,2)} \| F \|_{L_t^{2} (I; L_{\rho}^{a'} L_{\omega}^{b'}(|\cdot|^{a'\gamma}))}
 \| G\|_{L_t^{2} (J; L_{\rho}^{2} L_{\omega}^{\tilde{b}'}(|\cdot|^{2\gamma}))}$$
as desired.

\section{The well-posedness in $L^2$}\label{sec4}
In this section we prove Theorem \ref{cor} making use of the weighted Strichartz estimates in Theorems \ref{prop1} and \ref{Prop1}.

\subsection{Nonlinear estimates}
The following nonlinear estimates play a key role in the proof.

\begin{lem}\label{nonlem}
Let $n \ge 3$, $0< \alpha < 2$ and $ \beta =(4-2\alpha)/n$.
Assume that the exponents $(1/r,1/k)$ satisfy all the conditions given as in Theorem \ref{cor}.
Then there exist certain exponents $(1/r, 1/\tilde{k})$
in the open quadrangle with vertices $A,D,C,B$ in Figure \ref{fig1},
 for which
\begin{equation*}
\big\| |x|^{-\alpha} |u|^{\beta}v\big \|_{L_t^{2}([0,\infty); L_\rho^{r'}
L_\omega^{\tilde{k}'}(|\cdot|^{r'\gamma}) )}
\leq \| u \|_{L_{t}^{\infty}([0,\infty); L_x^2)}^{\beta}
\| v \|_{L_{t}^{2}([0,\infty);L_{\rho}^{r}L_{\omega}^{k}(|\cdot|^{-r\gamma}))}
\end{equation*}
holds with $\gamma=\alpha/2$.
\end{lem}

\begin{proof}[Proof of Lemma \ref{nonlem}]
Let us first write the conditions on $(r,k;\gamma)$ as
\begin{equation}\label{qpk}
 0 < \gamma < 1,\quad \gamma=1 -n\left(\frac{1}{2} - \frac{1}{r}\right),\quad 0< \frac{1}{r} - \frac{1}{k} < \frac{\gamma}{2(n-1)},
\end{equation}
and set, for $1<\tilde{k}<\infty$,
\begin{equation}\label{qpktilde}
\frac{1}{r} - \frac{1}{\tilde{k}}<0
\end{equation}
so that
$(1/r, 1/\tilde{k})$
lies in the open quadrangle with vertices $A,D,C,B$.

We then let
\begin{equation}\label{setting}
 \frac{1}{r'}=\frac{\beta}{2}+\frac{1}{r}, \quad \frac{1}{\tilde{k}'}=\frac{\beta}{2}+\frac{1}{k} \quad \textnormal{and} \quad \gamma-\alpha=-\gamma
\end{equation}
with which we use H\"older's inequality to obtain
\begin{align*}
	\begin{split}
\left\| |x|^{-\alpha} |u|^{\beta} v \right\|_{L_t^{2}([0,\infty) ; L_{\rho}^{r'}L_{\omega}^{\tilde{k}'}(|\cdot|^{r' \gamma}))}
& = \left\| |x|^{\gamma-\alpha} |u|^{\beta} v \right\|_{L_t^{2}([0,\infty) ; L_{\rho}^{r'}L_{\omega}^{\tilde{k}'})} \cr
& =  \left\| |x|^{-\gamma} |u|^{\beta} v \right\|_{L_t^{2}([0,\infty) ; L_{\rho}^{\frac{2r}{r\beta+2}}L_{\omega}^{\frac{2k}{k \beta+2}})} \cr
& \leq \left\| u \right\|_{L_{t}^{\infty}([0,\infty); L_{x}^{2})}^{\beta} \| v \|_{L_{t}^{2}([0,\infty); L_{\rho}^{r}L_{\omega}^{k}(|\cdot|^{-r\gamma}))}
	\end{split}
\end{align*}
as desired in the lemma.

Now we only need to check the condition \eqref{condition22}.
Using the first and second conditions of \eqref{setting},
it is not difficult to check that \eqref{qpktilde} may be replaced by the last one in \eqref{qpk}.
We insert $\gamma=\alpha/2$ from the last one in \eqref{setting} into \eqref{qpk}.
Then it follows that
\begin{equation}\label{44}
0<\alpha<2, \quad \frac{1}{r}=\frac{n-2}{2n}+\frac{\alpha}{2n},\quad \frac{1}{r}-\frac{\alpha}{4(n-1)}<\frac{1}{k} <\frac{1}{r}
\end{equation}
where the last two conditions are exactly the same as in \eqref{condition22}.
Note here that the second condition of \eqref{44} combined with the first one of \eqref{setting} implies $\beta=(4-2\alpha)/n$.
\end{proof}

\subsection{Contraction mapping}
By Duhamel's principle, the solution of the Cauchy problem \eqref{INLS} can be written as
\begin{equation}\label{Duhamel}
	\Phi(u) := e^{it \Delta} u_0 - i\lambda  \int_{0}^{t} e^{i(t-s) \Delta}F(u)ds
\end{equation}
where $F(u) = |\cdot|^{-\alpha} |u(\cdot,s)|^{\beta}u(\cdot, s)$.
For suitable values of $ M, N >0$, we suffice to show that $\Phi$ defines a contraction on
\begin{align*}
X(T,M, N) = \Big\{ u \in & C_t([0,\infty) ; L_x^2) \cap L_t^2 ([0,\infty) ; L_{\rho}^r L_{\omega}^{k}(|\cdot|^{- \alpha r/2}) : \cr
 & \sup_{t \in [0,\infty)} \| u \|_{L_x^2} \leq N,\ \| u \|_{L_t^2([0,\infty) ; L_{\rho}^r L_{\omega}^{k}(|\cdot|^{- \alpha r/2}))} \leq M\Big\}
\end{align*}
equipped with the metric
\begin{equation*}
d(u,v) = \sup_{t \in [0,\infty)} \| u-v \|_{L_x^2} + \| u- v \|_{L_t^2([0,\infty); L_{\rho}^r L_{\omega}^k(|\cdot|^{-\alpha r/2 }))}
\end{equation*}
where $(r,k)$ is given as in Theorem \ref{cor}.
To do so, we need some inhomogeneous estimates to control the Duhamel term in \eqref{Duhamel}.
In our case, the estimates in Theorem \ref{Prop1} are enough.

Now we show that $\Phi$ is well defined on $X$. In other words, for $ u \in X$
\begin{equation*}
\sup_{t \in [0,\infty)} \| \Phi(u) \|_{L_x^2}\leq N\quad\text{and}\quad \| \Phi(u) \|_{L_t^2([0,\infty); L_{\rho}^r L_{\omega}^k(|\cdot|^{- \alpha r/2}))} \leq M.
\end{equation*}
Using Plancherel's theorem, the adjoint version of \eqref{weighT} for $\gamma=\alpha/2$ combined with Remark \ref{rem}, and Lemma \ref{nonlem} in turn, we see
\begin{align*}
\sup_{t \in[0,\infty)} \| \Phi(u) \|_{L_x^2}
& \leq C \| u_0 \|_{L^2} + C \bigg\| \int_{-\infty}^{\infty} e^{-is\Delta} \chi_{[0,t]}(s) F(u)ds \bigg\|_{L_x^2} \cr
& \leq C \| u_0 \|_{L^2} + C \| F(u) \|_{L_t^{2}([0,\infty); L_\rho^{r'} L_\omega^{\tilde{k}'}(|\cdot|^{ \alpha r' /2}))} \cr
& \leq C \| u_0 \|_{L^2} + CN^{\beta}M.
\end{align*}
On the other hand,
applying \eqref{weighTT*} with $\gamma=\alpha/2$ to \eqref{Duhamel}, and then using Lemma \ref{nonlem}, we obtain
\begin{align*}
\| \Phi(u) &\|_{L_t^2([0,\infty); L_{\rho}^r L_{\omega}^k(|\cdot|^{- \alpha r/2 }))}\\
 &\leq \| e^{it\Delta}u_0 \|_{L_t^2([0,\infty) ; L_{\rho}^r L_{\omega}^{k}(|\cdot|^{- \alpha r/2}))} + C \| F(u) \|_{L_t^{2} ( [0,\infty); L_\rho^{r'}L_\omega^{\tilde{k}'}(|\cdot|^{ \alpha r'/2})}\\
&\leq \| e^{it\Delta}u_0 \|_{L_t^2([0,\infty); L_{\rho}^r L_{\omega}^{k}(|\cdot|^{- \alpha r/2}))} +  C N^{\beta}M.
\end{align*}
Here we observe that
\begin{equation*}
\| e^{it\Delta}u_0 \|_{L_t^2([0,\infty); L_{\rho}^r L_{\omega}^{k}(|\cdot|^{- \alpha r/2}))} \leq \varepsilon
\end{equation*}
for some sufficiently small $\varepsilon > 0$ chosen later, if  $\| u_0 \|_{L^2}$ is small (see \eqref{weighT} with $\gamma=\alpha/2$).
We therefore get $ \Phi(u) \in X$ for $u \in X$ if
\begin{equation}\label{pop}
C\| u_0 \|_{L^2} + CN^{\beta}M \leq N \quad \textnormal{and} \quad \varepsilon + CN^{\beta}M \leq M.
\end{equation}

Next we show that $\Phi$ is a contraction. Namely, for $u,v \in X$
\begin{equation*}
d(\Phi(u), \Phi(v)) \leq \frac{1}{2}d(u,v).
\end{equation*}
As before, we see
\begin{align*}
d(\Phi(u), \Phi(v)) & = \sup_{t \in [0,\infty)} \| \Phi(u) - \Phi(v) \|_{L_x^2} + \| \Phi(u) - \Phi(v) \|_{L_t^2([0,\infty); L_{\rho}^r L_{\omega}^{k}(|\cdot|^{- \alpha r/2}))} \cr
&\leq C \| F(u) - F(v) \|_{L_t^{2}([0,\infty); L_\rho^{r'} L_\omega^{\tilde{k}'}(|\cdot|^{\alpha r'/2}))}.
\end{align*}
We will show
\begin{equation*}
\| F(u) - F(v) \|_{L_t^{2}([0,\infty); L_\rho^{r'} L_\omega^{\tilde{k}'}(|\cdot|^{ \alpha r'/2}))}
\leq CN^\beta \| u-v \|_{L_{t}^{2}([0,\infty); L_{\rho}^{r}L_{\omega}^{k}(|\cdot|^{-\alpha r/2}))}
\end{equation*}
which is reduced to showing
\begin{equation*}
\left\| |x|^{-\alpha} |u|^{\beta} |u-v| \right\|_{L_t^{2}([0,\infty); L_{\rho}^{r'} L_{\omega}^{\tilde{k}'}(|\cdot|^{ \alpha r'/2}))} \leq N^\beta \left\| u-v \right\|_{L_t^{2}([0,\infty); L_{\rho}^{r} L_{\omega}^{k}(|\cdot|^{- \alpha r/2}))}
\end{equation*}
and
\begin{equation}\label{bysym22}
\| |x|^{-\alpha} |v|^{\beta}|u-v|\|_{L_t^{2}([0,\infty); L_{\rho}^{r'} L_{\omega}^{\tilde{k}'}(|\cdot|^{\alpha r'/2}))} \leq N^{\beta}\| u-v \|_{L_{t}^{2}([0,\infty); L_{\rho}^{r}L_{\omega}^{k}(|\cdot|^{- \alpha r/2}))}.
\end{equation}
by the following simple inequality
\begin{equation*}
(|u|^{\beta}u - |v|^{\beta}v) \leq C (|u|^{\beta}+|v|^{\beta})|u-v|.
\end{equation*}
We apply Lemma \ref{nonlem} with $v$ replaced by $|u-v|$ to obtain
\begin{align*}
\| |x|^{-\alpha} |u|^{\beta} &|u-v|\|_{L_t^{2}([0,\infty); L_{\rho}^{r'} L_{\omega}^{\tilde{k}'}(|\cdot|^{\alpha r'/2}))} \cr
&\leq  \| u \|_{L_{t}^{\infty}([0,\infty); L_{x}^{2})}^{\beta} \| u-v \|_{L_{t}^{2}([0,\infty); L_{\rho}^{r}L_{\omega}^{k}(|\cdot|^{-\alpha r/2}))} \cr
& \leq N^{\beta}\| u-v \|_{L_{t}^{2}([0,\infty); L_{\rho}^{r}L_{\omega}^{k}(|\cdot|^{-\alpha r/2}))}.
\end{align*}
Similarly, we get \eqref{bysym22}.
Hence we obtain
\begin{align*}
d(\Phi(u), \Phi(v)) &\leq CN^{\beta} \| u-v \|_{L_t^2([0,\infty) ; L_{\rho}^r L_{\omega}^{k}(|\cdot|^{- \alpha r/2}))} \cr
& \leq CN^{\beta} d(u,v).
\end{align*}
Now by taking $N = 2C\| u_0 \|_{L^2}$ and $M = 2\varepsilon$
and then choosing $\varepsilon >0$ small enough such that \eqref{pop} holds and $CN^\beta \leq 1/2$,
it follows that $\Phi$ is a contraction on $X$.

Finally, we show the scattering property.
Using \eqref{Duhamel} and following the argument above, one can easily see that
\begin{align*}
\| e^{-it_2 \Delta} u(t_2) - e^{-it_1 \Delta} u(t_1) \|_{L_x^2} & = \bigg\| \int_{t_1}^{t_2} e^{-is\Delta} F(u)ds \bigg\|_{L_x^2} \cr
& \lesssim \| F(u) \|_{L_t^{2}([t_1,t_2] ; L_{\rho}^{r'} L_{\omega}^{\tilde{k}'}(|\cdot|^{ \alpha r'/2}))} \cr
& \lesssim \| u \|_{L_t^{\infty}([t_1,t_2] ; L_{x}^2)}^{\beta} \| u \|_{L_t^2([t_1,t_2] ; L_{\rho}^r L_{\omega}^{k}(|\cdot|^{- \alpha r/2}))}\quad\rightarrow\quad0
\end{align*}
as $t_1, t_2 \rightarrow \infty$.
This yields that $$ \varphi := \lim_{t \rightarrow \infty} e^{-it\Delta} u(t)$$ exists in $L^2$.
In addition, one has $$ u(t) - e^{it\Delta}\varphi = i\lambda \int_{t}^{\infty} e^{i(t-s)\Delta}F(u)ds, $$ and therefore
\begin{align*}
\| u(t) - e^{it\Delta}\varphi \|_{L_x^2} &= \bigg\| \int_{t}^{\infty} e^{i(t-s)\Delta} F(u) ds \bigg\|_{L_x^2}\\
& \lesssim \| F(u) \|_{L_t^{2}([t, \infty) ; L_{x}^{r'} L_{\omega}^{\tilde{k}'}(|\cdot|^{\alpha r'/2}))} \\
& \lesssim \| u \|_{L_t^{\infty}([t, \infty) ; L_{x}^2)}^{\beta} \| u \|_{L_t^2([t, \infty) ; L_{\rho}^r L_{\omega}^{k}(|\cdot|^{- \alpha r/2}))}\quad\rightarrow\quad0
\end{align*}
as $ t \rightarrow \infty$. This completes the proof.

\subsubsection*{Acknowledgment.} The authors would like to thank Yonggeun Cho for informing us
of the paper \cite{GHN}.

\subsubsection*{Data availability statement.}
Data sharing not applicable to this article as no datasets were generated or analysed during the current study.


\end{document}